\documentclass[11pt,reqno]{amsart}
\usepackage{amssymb,amsmath,latexsym}
\usepackage{graphicx}
\usepackage{amsmath}
\usepackage{amsfonts}
\usepackage{amssymb}

\newtheorem{theorem}{Theorem}
\newtheorem{proposition}[theorem]{Proposition}

\theoremstyle{remark}
\newtheorem{remark}[theorem]{Remark}

\begin{document}

\title[Embedding algorithms and applications to differential equations]{Embedding 
algorithms and applications to differential equations}

\author[S. Ali]{Sajid Ali}

\address{Department of Basic Sciences, School of Electrical Engineering and Computer
Science, National University of Sciences and Technology, Islamabad 44000, Pakistan}

\email{sajid$\_$ali@mail.com}

\author[H. Azad]{Hassan Azad}

\address{Department of Mathematics and Statistics, King Fahd University,
Saudi Arabia}

\email{hassanaz@kfupm.edu.sa}

\author[I. Biswas]{Indranil Biswas}

\address{School of Mathematics, Tata Institute of Fundamental Research, Homi
Bhabha Road, Mumbai 400005, India}

\email{indranil@math.tifr.res.in}

\author[R. Ghanam]{Ryad Ghanam}

\address{Virginia Commonwealth University in Qatar, Education City
Doha, Qatar}

\email{raghanam@vcu.edu}

\author[M.T. Mustafa]{M. T. Mustafa}

\address{Department of Mathematics, Statistics and Physics, Qatar
University, Doha, 2713, State of Qatar}

\email{tahir.mustafa@qu.edu.qa}



\date{}

\begin{abstract}
Algorithms for embedding certain types of nilpotent subalgebras in maximal subalgebras 
of the same type are developed, using methods of real algebraic groups. These 
algorithms are applied to determine non-conjugate subalgebras of the symmetry algebra 
of the wave equation, which in turn are used to determine a large class of invariant 
solutions of the wave equation. The algorithms are also illustrated for the symmetry 
algebra of a classical system of differential equations considered by Cartan in the 
context of contact geometry.
\end{abstract}

\maketitle

\section{Introduction}

One of the main applications of Lie algebras is to find solutions of differential 
equations, by reduction of order, or by using conjugacy classes of its subalgebras to 
find invariant solutions. The method of invariant solutions goes back to
\cite[Ch.~X]{Lie}. 
This method is also explained in detail in the books of Ibragimov \cite[Ch~9]{ib}, 
Ibragimov \cite{ibsw}, Bluman \cite{bluman1, bluman2} and Olver \cite[Ch~3]{Ol}.

The Lie theoretic input in this method is a list of conjugacy classes of subalgebras 
of dimension depending on the order of the equation. A detailed structure of the 
symmetry algebra is also useful in finding linearizing coordinates for linearizable 
equations.

It is our experience, based on \cite{ADGM}, that if the algebras are not chosen
appropriately, 
they are practically useless, because a preliminary step is to find their invariants 
and there is no algorithmic procedure to do that. However, if the subalgebras are 
constructed from the geometry of the space on which one is studying a given equation, 
for example by embedding translations or scalings in maximal subalgebras, the 
characteristics of the subalgebras obtained are manageable.

The principal aim of this note is to give algorithms for embedding given abelian and 
solvable algebras of certain types in maximal subalgebras of the same type, using 
standard commands of Maple.

The precise types of the subalgebras are given in the algorithms constructed below.

Maple is able to find the Cartan decomposition as well as root space decompositions 
for semisimple algebras of fairly high dimensions. The algorithms it uses are based on 
the fundamental papers of Rand, Winternitz and Zassenhaus \cite{RWZ}, of de Graff 
\cite{Dg}, and of Dietrich, Faccin and de Graaf \cite{DFG}
and Ian	Anderson \cite{An}. The recent book of 
\v{S}nobl-Winternitz \cite{SW} gives a detailed account of some of these algorithms.

Derksen, Jeandel and Koiran \cite{DJK} have also developed algorithms for
computing the Zariski closures of linear solvable Lie groups and the
algorithms in this paper reduce the computation of Zariski closures of
linear groups to those of abelian subgroups.
The algorithms given in this paper are based on results of Mostow \cite{Mo}
on real algebraic groups; see a recent account of the subject in \cite{AB}.
All the algebras considered in this paper are assumed to be real
algebraic Lie algebras. We recall their definition and some basic facts
about them.
\begin{enumerate}
\item Let ${\mathfrak g}\, \subset\, {\rm gl}(n, {\mathbb R})$ be a Lie algebra, with
$G$ the corresponding Lie
group. The algebra ${\mathfrak g}$ is called an algebraic Lie algebra if the
group $G^{\mathbb C}\,\subset\, {\rm GL}(n, {\mathbb C})$ with Lie algebra
${\mathfrak g}^{\mathbb C}\,=\, {\mathfrak g} \oplus \sqrt{-1}{\mathfrak g}$ is an
algebraic group in as defined in \cite{Bo}.

\item{} The group $G^{\mathbb C}$ is generated by complex $1$-parameter subgroups
$$\{\exp (zX) \,\mid\, z\,\in\, {\mathbb C}\},\ \ X \,\in\, {\mathfrak g}$$
and the connected component of the real points of $G^{\mathbb C}$ is the group $G$.
\end{enumerate}

All real semisimple Lie algebras, all real linear Lie algebras generated
by nilpotents, all abelian real linear algebras of semisimple elements
defined by integral equations as well as all subalgebras generated by
the types of algebras already listed are examples of algebraic Lie
algebras.

If an abstract Lie algebra is given by its commutator table, then one is
tacitly working in its adjoint representation and all the concepts
regarding semisimplicity, nilpotency etc. are with reference to this
representation.

As regards the reduced root space decompositions, Maple needs a
maximal split abelian algebra of semisimple elements. If that is not
specified, the commands will give, in general, roots taking complex
eigen-values.

The algorithms constructed in this paper give both the relative and absolute
root system of the Lie algebra more or less automatically.

We will illustrate the algorithms by working out in detail an embedding of a 
subalgebra, which is clearly abelian, of the symmetry algebra of the wave equation on 
flat 4 dimensional space, in a maximal solvable subalgebra: this gives at the same 
time detailed structure of the symmetry algebra and several non-conjugate subalgebras.

These algebras are non-conjugate by construction. A solvable linear algebraic 
algebra is itself a sum of a torus and its maximal nilpotent ideal. A linear 
algebraic torus has either all real eigen-values or all purely imaginary eigen-values 
or it is itself a sum of two such tori \cite{Bo}, \cite[Proposition~1]{AB}. 
Thus if given subalgebras of a semisimple algebra are abstractly isomorphic and 
solvable one looks at their semisimple and nilpotent parts in the adjoint 
representation of any semisimple subalgebra containing them to decide if they are 
non-conjugate. The uniqueness of the Jordan decomposition \cite{Bo} ensures 
that it is immaterial which ambient subalgebra one chooses as long as it is 
semisimple.

If the subalgebras are semisimple and abstractly isomorphic, one looks 
at their centralizers or normalizers to decide if they are non-conjugate.

It may happen that their centralizers or normalizers are the same. In that case 
one has to use Bruhat decomposition or its variants \cite{Kn} or methods similar 
to \cite{GW}, \cite{PWZ} to decide non-conjugacy under the adjoint group.
This type of complete symmetry analysis has not been done in this paper and the 
examples were chosen to illustrate that while doing the algorithms interactively 
and computing normalizers or centralizers one gets several subalgebras for which 
one could check their non-conjugacy without using Bruhat or Iwasawa decomposition.

We do the same for the symmetry algebra of a classical nonlinear ODE, whose 
symmetry algebra was determined and identified as the exceptional algebra $G_2$ by 
Cartan --- see \cite{Agricola} for further references, and also \cite{AKO}
and \cite{Ke}. The 
group $G_2$ also has a very interesting relation to mechanics. A recent account is 
in Bor and Montgomery \cite{BM}. One of the main points in \cite{BM} is to 
identify a maximal compact subalgebra of $G_2$ and give its explicit 
decomposition. The identification and decomposition of the maximal compact 
subalgebra of $G_2$ follows from algorithm 1 in a straight forward way.

The structural information obtained in Section 4 gives several three and four dimensional 
subalgebras that are non-conjugate in the adjoint representation. This gives a more extensive and 
useable list of subalgebras than that given in \cite{ADGM} and these subalgebras are used in the 
last section to give solutions of the wave equation in flat 4-d space.

In a follow up of this paper, a similar analysis of the wave equation on all static spherically symmetric spaces times will be given to obtain a more extensive list of solutions than that given in \cite{ADGM} of the wave equation on certain spherically symmetric four dimensional spaces.
In particular, such invariant solutions will be given for all types of $3$ dimensional subalgebras that can arise as subalgebras of the symmetry algebra of the equation.

The reader is referred to \cite{ibsw} for very general results on the wave equation on Riemannian manifolds and several other equations of relevance to physics.

As far as Lie algebras are concerned, we need the following results:
\begin{enumerate}
\item If $X$ is an element of a Lie algebra $L$ which is the Lie algebra of a real 
algebraic subgroup of ${\rm GL}(n,{\mathbb R})$, and $X\,=\, X_s+X_n$ is the 
Jordan decomposition of $X$, then both $X_s$ and $X_n$ are in $L$ \cite[p.~14, 
Section~3.7, Proposition~1, Lemma~1]{Bo}, \cite{AB}. In fact, they are in the center of the 
centralizer of $X$ in $L$.

\item If $S$ is an abelian subalgebra consisting of semisimple elements of a real 
algebraic Lie algebra, then its centralizer $Z(S)$ has the Levi decomposition
$$
Z(S) \,=\, [Z(S),\, Z(S)] \oplus Z(Z(S))\, ,
$$
where $Z(Z(S))$ is the center of $Z(S)$ \cite{BT} (a proof of this is also given below
in Proposition \ref{propL}).

\item The derived algebra of any solvable algebra is ad-nilpotent. In particular, 
if $H$ is any subalgebra of $L$ then the derived algebras of the radical of 
centralizer of $H$ and of the normalizer of $H$ are ad-nilpotent
\cite[p.~105, Theorem~5.4.7]{HN}, \cite[Proposition~1]{AB}.

\item If $H$ is a semisimple subalgebra of $L$ and $X$ is an ad-semisimple or ad-nilpotent 
element of $H$ in the adjoint representation of $H$ on itself, then $X$ is also 
ad-semisimple or ad-nilpotent in the adjoint representation of $L$ \cite[p.~14,
3.7]{Bo}.
\end{enumerate}

These facts are very useful in verifying that a certain element is semisimple or 
nilpotent by reducing the computations to subalgebras of small dimensions.

\section{Roots}\label{sR}

\subsection{Roots of a semisimple algebra}

A few words regarding the section on roots are in order. Maple will give -- for any 
Cartan algebra -- an array of complex numbers. In the following section we explain 
how to extract a simple system of roots and the corresponding Dynkin diagram 
directly from such a list.

Let $L$ be a semisimple Lie subalgebra of $\mathfrak{g}\mathfrak{l}(n,
{\mathbb R})$ and $C$ a Cartan subalgebra of $L$. The algebra $C$
is, by definition, a maximal abelian subalgebra of diagonalizable
elements in the complexification of $L$. A nonzero vector $v$ in
${L}\oplus\sqrt{-1}L$ such that
$$
[h\, ,v] \,=\, \lambda(h)\cdot v
$$
for all $h\, \in\, C$ is called a \textit{root vector} and the corresponding linear
functional $\lambda$ is called a root of the Cartan algebra $C$.

In general, the roots will be complex valued, so one needs to define what it means
for a complex valued root to be positive -- based only on the list of roots provided
by the program. This is sufficient to describe the Dynkin diagram algorithmically, as
detailed below.

A complex number $z\,=\, a+\sqrt{-1}b$, where $a\, , b\, \in\, \mathbb R$, is positive
if either its real part $a$ is positive or $a\,=\, 0$ but $b\, >\,0$.

Fix a basis $h_1\, , \cdots\, , h_r$ of $C$. A non-zero root
$\lambda$ is positive if the first nonzero number $\lambda(h_i)$ is
a complex positive number. Otherwise, it is called a negative root.

Positive roots which are not a sum of two positive roots are called simple roots.

For sake of convenience, henceforth a root will mean a non-zero
root.

\subsection{Restricted roots}\label{secrr}

An abelian subalgebra of $L$ consisting of semisimple elements in the adjoint
representation on $L$, is,
by definition, a {\it torus}. If, moreover, all its elements in the adjoint representation of
$L$ have real eigen-values, then it is a {\it real} torus; if all eigen-values
are purely imaginary, it is called a {\it compact} torus.

Any real algebraic torus is a sum of a real and a compact torus and the dimensions 
its real and compact parts are invariants of the torus \cite[Proposition~1]{AB}.

Moreover, all maximal solvable subalgebras $B$ of a real semisimple algebra with real
eigen-values in the adjoint representation are conjugate \cite{Mo}, \cite{AB}.
In the context of the Iwasawa decomposition \cite{Kn}, \cite{HN}, this is the 
algebra $A\oplus N$.
	
If $A$ is a maximal torus of B then the full algebra is a sum of $A$-invariant 
subspaces --- of dimension possibly greater than one --- and the roots $A$ in $B$ which are 
not a sum of two roots in $B$ are simple roots of a root system --- in the sense of 
\cite{HN}, \cite{Kn}. In case that the real semisimple algebra has a maximal torus with all real 
eigen-values we can define positive roots without going to the complexification of 
the algebra. In this case the positive root spaces together with the torus give a maximal 
solvable algebra whose eigen-values are all real and all such algebras are conjugate. Therefore,
the restricted root system and the absolute root system coincide in this case.
	
For each positive simple root $\alpha$ we can find a standard set of generators $X_\alpha,\,
Y_\alpha,\, H_\alpha$ with $X_\alpha,\, Y_\alpha$ eigen-vectors of ${\rm ad}(H_\alpha)$ with
opposite and nonzero eigen-values. This three dimensional subalgebra is therefore isomorphic
to ${\rm sl}(2, {\mathbb R})$. If, for each simple root $\alpha$ we fix an isomorphism
$$\varphi_\alpha\, :\, {\rm sl}(2, {\mathbb R}) \,\longrightarrow\, 
\langle X_\alpha,\, Y_\alpha,\, H_\alpha\rangle\, ,
$$
where $\langle X_\alpha,\, Y_\alpha,\, H_\alpha\rangle$ is the Lie subalgebra generated
by $X_\alpha,\, Y_\alpha,\, H_\alpha$, then the elements
$$
\varphi_\alpha\begin{pmatrix}0 &1\\ -1 & 0\end{pmatrix}
$$
generate, as a Lie subalgebra, a maximal compact subalgebra of the given Lie 
algebra --- only in the case that the Lie algebra has a maximal torus whose 
eigen-values are all real \cite[p.~100, Lemma~43]{St}.

\subsection{Procedure for constructing Dynkin diagram}

Positive roots which are not a sum of two positive roots are called simple roots. 
Thus to obtain simple roots from a given set of positive roots, one adds pairs of 
positive roots and marks those that are sums of positive roots; at the end, one 
strikes out those roots that are sums of two positive roots and the remaining ones 
will be simple roots.

Let $a$, $b$ be simple roots. The positive roots among the integral combinations of 
them determine the bond between $a$ and $b$. The simple roots $a$, $b$ are not 
joined if $a+b$ is not a root. They are joined by a single bond if $a$, $b$ and 
$a+b$ are the only positive roots among the integral combinations of $a$ and $b$. 
They are joined by a double bond with arrow pointing from $a$ to $b$ if $a$, $b$, 
$a+b$ and $a+2b$ are the only positive roots among the integral combinations of $a$ 
and $b$. They are joined by a triple bond with arrow pointing from $a$ to $b$ if 
$a$, $b$, $a+b$, $a+2b$, $a+3b$ and $2a+3b$ are the only positive roots among the 
integral combinations of $a$ and $b$.

The diagram then identifies the complexification of the Lie algebra $L$.

\section{Algorithms}\label{aL}

We need the following result on centralizers of semisimple
elements to implement the algorithms.

\begin{proposition}\label{propL}
Let $S$ be a commuting algebra of diagonalizable
elements in a real semisimple algebra $L$. Then the centralizer 
of $S$ has the Levi decomposition
$$
Z_L(S)\,=\, [Z_L(S),\, Z_L(S)]\oplus Z(Z_L(S))\, ,
$$
where $Z_L(S)$ is the centralizer of $S$ in $L$, and
$Z(Z_L(S))\, \subset\, Z_L(S)$ is its center.
\end{proposition}

\begin{proof}
Recall that we defined a nonzero complex number to be
positive if its real part is positive or if its real part is zero and the
imaginary part is positive.

Include $S$ in a maximal torus $T$. The complexification of $T$ is a
maximal torus of the complexification $L^{\mathbb C}$ of $L$
\cite[Corollary~7]{AB}, and the centralizer
$Z_L(S)$ of $S$ in $L$ is the same as the real points of the centralizer of
$S^{\mathbb C}$ in $L^{\mathbb C}$. Now $Z_{L^\mathbb{C}}(S^{\mathbb C})$ is
generated by $T^{\mathbb C}$ and the root vectors $X_\alpha$
such that $\alpha(S)\,=\, 0$. This is a closed set of roots and the Lie
algebra $Z_{L^\mathbb{C}}(S^{\mathbb C})$ contains the root vector $X_{-\alpha}$
for every $\alpha$ as above. By
extending a basis of $S^{\mathbb C}$ say $\{s_i\}_{i=1}^m$
to a basis of $T^{\mathbb C}$, say $\{s_i\}_{i=1}^n$ and declaring a root
$r$ to be positive if the
first nonzero number $r(s_i)$ is positive, we see that the indecomposable
positive roots of $Z_{L^\mathbb{C}}(S^{\mathbb C})$ are simple roots of
$T^{\mathbb C}$. Thus these generate a semisimple subalgebra $L_1$ of
$Z_{L^\mathbb{C}}(S^{\mathbb C})$ and
$Z_{L^\mathbb{C}}(S^{\mathbb C}) \,=\, \langle L_1,\, T^{\mathbb C}\rangle$. Since
$T^{\mathbb C}$ normalizes $L_1$ the commutator is
$L_1$. Hence $Z_{L^\mathbb{C}}(S^{\mathbb C})/L_1$, being an image of
$T^{\mathbb C}$, contains no nilpotents.

Therefore, the Levi decomposition of $Z_{L^\mathbb{C}}(S^{\mathbb C})$ is 
$$
Z_{L^\mathbb{C}}(S^{\mathbb C})\,=\, L_1+ R\, ,
$$
and $R$, being solvable with no nilpotents, is a torus. As $[L_1,\, R]$ is
contained in both $L_1$ and $R$ it must be $0$. Hence $R$ is a central
torus and it is equal to $Z(Z_{L^\mathbb{C}}(S^{\mathbb C}))$ --- the center of
$Z_{L^\mathbb{C}}(S^{\mathbb C})$.
Therefore, the Levi decomposition of $Z_{L^\mathbb{C}}(S^{\mathbb C})$ is 
$[Z_{L^\mathbb{C}}(S^{\mathbb C}),\, Z_{L^\mathbb{C}}(S^{\mathbb C})]
\oplus Z(Z_{L^\mathbb{C}}(S^{\mathbb C}))$.
Taking real points gives the Levi decomposition
$Z_L(S)\,=\, [Z_L(S),\, Z_L(S)]\oplus Z(Z_L(S))$.
\end{proof}

The algorithms given below are similar to each other. For the
convenience of the user we have written down complete details-at the
expense of repetition-of the most frequent types of algebras
encountered in practice.

\subsection{Algorithm for embedding a given abelian subalgebra of semisimple elements
with real eigen-values}

Here we give an algorithm for embedding a given abelian subalgebra of semisimple elements
with real eigen-values in a maximal algebra of such elements and in a maximally real
Cartan algebra: in this algorithm, the ambient algebra is assumed to be semisimple.

For a subalgebra $H$ of $L$, let $N_L(H)$, $Z_L(H)$, $Z(H)$ and $H'$
denote its normalizer in $L$, its centralizer in $L$, its center and
its derived algebra respectively. For notational convenience, we
will also write $N(H)$ for $N_L(H)$.

Let $A$ be real torus (defined in Section \ref{secrr}).

\textbf{Step 1:}\, Compute $Z_L(A)$, the centralizer of $A$ in $L$,
the derived algebra $Z_L(A)'$ of $Z_L(A)$ and the center $Z(Z_L(A))$
of $Z_L(A)$. Then one has the direct sum decomposition
$$
Z_L(A)\,=\, Z_L(A)'\oplus Z(Z_L(A))\, .
$$

\textbf{Step 2:}\, Compute the Killing form of $Z_L(A)'$. If it is negative definite, then the real
part of the subalgebra $Z(Z_L(A))$ is a maximal real torus.

\textbf{Step 3:}\, If the Killing form of $Z_L(A)'$ is indefinite,
compute the Cartan decomposition of $Z_L(A)'$, and pick any nonzero
element from the radial part of the decomposition and adjoin it to
$A$.

Repeat Step 1 and Step 2 till an abelian algebra, which we again denote by $A$, is obtained
which has all real eigen-values in the adjoint representation -- and in the decomposition
$Z_L(A)\,=\, Z_L(A)'\oplus Z(Z_L(A))$, the Killing form of $Z_L(A)'$ is negative definite.

At this stage, a maximal real torus containing the given algebra has been obtained.
Denote it again by $A$.

The compact part of $Z(Z_L(A))$ together with a maximal torus of $Z_L(A)'$ is a compact torus.
Adjoining it to $A$ gives a maximally real Cartan algebra.

\begin{remark}
By an entirely similar procedure, a compact torus can be embedded in a maximally compact
Cartan subalgebra.
\end{remark}

\subsection{Algorithm for embedding a commutative subalgebra of
ad-nilpotent elements to a maximal commutative subalgebra of such
elements}

Let $U$ be an abelian algebra of ad--nilpotent elements. Let
$$
Z_L(U)\,=\, S\oplus R
$$
be the Levi decomposition of $Z_L(U)$. Compute the derived
subalgebra $R'\, \subset\, R$. If $\dim (R'+U)\, >\, \dim U$, adjoin
any element of $R'$ complementary to $U$ to obtain a commutative subalgebra of
ad--nilpotent elements. Repeat this procedure until an abelian
algebra of ad--nilpotent elements is obtained -- which we denote
again by $U$ -- so that in the Levi decomposition on $Z_L(U)$, the
algebra $R'$ is contained in $U$.

At this stage if $R$ contains $U$ as a proper subalgebra, consider an element $x$ 
of $R$ complementary to $U$. Then the nilpotent and semisimple parts of $x$ belong to 
$R$. If $x$ has a nonzero nilpotent part $x_n$, then the subalgebra generated by $U$ 
and $x_n$ is commutative consisting of ad--nilpotent elements.

Repeating the above procedure, we may assume that $U$ is a commutative
subalgebra of ad--nilpotent elements such that in the Levi decomposition
$$
Z_L(U) \,=\, S\oplus R\, ,
$$
$R'\, \subset\, U$, and every element in a basis of $R$ complementary to $U$ consists of
semisimple elements.

If the Killing form of $S$ is not negative definite, then $S$ will
have a nontrivial Cartan decomposition. Take an element $\alpha$ in
the radial part of the Cartan decomposition of $S$. As $S$ has no
center, the endomorphism ${\rm ad}(\alpha)$ of $S$ has a nonzero
real eigen-value. In fact any element all of whose eigen-values are
real will do. If $u$ is a nonzero eigen-vector of ${\rm ad}(\alpha)$ for such a nonzero
real eigen-value, then ${\rm ad}(u)$ is nilpotent on $S$. The reason is that $S$ is a
direct sum of eigen-spaces for ${\rm ad}(\alpha)$
and if $S_a$,$S_b$ are two such eigen-spaces, then $[S_a,\, S_b]$
is contained in $S_{a+b}$; therefore a sufficiently high power of ${\rm ad}(u)$ 
annihilates $S$. This implies that ${\rm ad}(u)$ is also nilpotent on $L$, by
uniqueness of Jordan decomposition. Adjoin $u$ to $U$ to obtain a higher
dimensional commutative algebra of nilpotents.

Thus repeating the above procedure we ultimately have a commutative subalgebra consisting
of ad--nilpotent elements, which we again denote by $U$, such that
$$
Z_L(U)\,=\, S+R\, ,
$$
$R'\, \subset\, U$, every element in a basis of $R$ complementary to $U$ consists of semisimple
elements and $S$ has a negative definite Killing form.

At this stage $U$ is a maximal abelian subalgebra of ad--nilpotent elements.

A very similar argument (-detailed below-) gives the embedding of a
given ad--nilpotent subalgebra in a maximal ad--nilpotent
subalgebra. Here one is assured of conjugacy of these subalgebras
\cite{AB}. Also, its normalizer will pick up a torus whose
eigen-values are all real and which is maximal with these properties.
Embedding this in a maximal torus -- using Algorithm 3.1 -- one will
obtain a maximally split Cartan subalgebra of $L$.

\subsection{Algorithm for embedding a subalgebra $U$ of ad--nilpotent elements to a 
maximal subalgebra of such elements}\mbox{}

\textbf{Step 1:}\, Find a normalizer of $U$ and compute its Levi decomposition
$$
N(U)\,=\, S\oplus R\, ,
$$
where $R$ is the radical.

\textbf{Step 2:}\, Compute the derived algebra $R'$ of $R$. If $\dim (R'+U)\, >\, \dim
U$, then this is again an ad--nilpotent algebra.

Repeat Step 1 and Step 2 so that ultimately $R'+U\,=\, U$. At this stage $R/U$ is abelian.

\textbf{Step 3:}\, We want to enlarge $U$ further so that $R/U$ consists entirely of semisimple
elements. To do this, take a basis of $U$, say $u_1\, , \cdots\, ,u_k$, and enlarge it to a basis
of $R$ by adjoining $v_1\, , \cdots\, ,v_\ell$. Find the Jordan decomposition of all
$v_1\, , \cdots\, ,v_\ell$. Adjoin to $U$ the nilpotent parts of all the $v_1\, , \cdots\, ,v_\ell$.
Denote by $\widetilde U$ the algebra obtained this way.

Now repeat Step 1, Step 2 and Step 3 till we have the Levi decomposition
$$
N(\widetilde U) \,=\, S+R
$$
such that $R'\, \subset\, \widetilde U$ and $R/\widetilde U$ consists only of
semisimple elements.

\textbf{Step 4:}\, If the Killing form of $S$ is not negative
definite, then it will have a nontrivial Cartan decomposition. Take
an element $\alpha$ in the radial part of the Cartan decomposition
of $S$. As the center of $S$ is trivial, the endomorphism ${\rm
ad}(\alpha)$ of $S$ has a nonzero real eigen-value. Let $u$ be a
nonzero eigen-vector for such an eigen-value. Then ${\rm ad}(u)$ is
nilpotent on $S$ and therefore on the Lie algebra $L$ -- by
uniqueness of the Jordan decomposition. Adjoin $u$ to $U$ to obtain
a higher dimensional algebra of nilpotents.

Iterating this procedure, we finally reach the situation that we have an ad--nilpotent
subalgebra, which we denote again by $U$, that contains the original ad--nilpotent subalgebra,
such that the Levi decomposition of $N(U)$ is
$$
N(U)\,=\, S\oplus R\, ,
$$
where $S$ has a negative definite Killing form, $R'\, \subset\, U$ and $R/U$ consists only
of semisimple elements in the sense that if we extend a basis of $U$ to a basis of $R$ and
find the Jordan decomposition of the basis elements outside $U$, then the nilpotent parts
all belong to $U$ \cite{AB}.

At this stage, $U$ is a maximal ad--nilpotent algebra containing the given ad--nilpotent algebra.

Finally, the abelian algebra representing $R/U$ is a torus and its real part is a maximal abelian
algebra consisting of real semisimple elements \cite{AB}. Denote this algebra by $A$. Then $A$ can
be enlarged to a maximally split Cartan algebra of the whole algebra, using Algorithm 3.1. This
algebra permutes the common eigen-spaces of $A$.

\section{Applications to structure of symmetry algebras related to certain equations of
physics}\label{Se4}

Before giving applications to symmetry algebras of higher dimensions, we illustrate
the algorithms to obtain structural information for the algebras ${\mathfrak
s}{\mathfrak o}(4)$, ${\mathfrak s}{\mathfrak o}(1,3)$ and ${\mathfrak s}{\mathfrak
o}(2,2)$.

Recall that if $A$ is an $n\times n$ diagonal matrix with diagonal
entries $1$, $-1$, then the Lie algebra of the corresponding
orthogonal group has generators $e_{ij}-e_{ji}$ if $a_{ii}a_{jj}
\,=\, 1$, and $e_{ij}+e_{ji}$ if $a_{ii}a_{jj} \,=\, -1$. Moreover,
if $A$ has the first $p$ diagonal entries $1$, and the next $q$
diagonal entries $-1$, where $p+q\,=\, n$, then the Lie algebra of
the corresponding orthogonal group is generated by
$$
e_{1,2}-e_{2,1}\, , \cdots\, , e_{p-1,p}-e_{p,p-1}\, , e_{p,p+1}+e_{p+1,p}\, ,
e_{p+1,p+2}-e_{p+2, p+1}\, , \cdots\, , e_{n-1,n}-e_{n,n-1}\, .
$$

\subsection{${\mathfrak s}{\mathfrak o}(4)$}\label{sec4.1l}

We now derive the factorization of ${\mathfrak s}{\mathfrak o}(4)$ using roots of
its maximal torus.

A basis of $V\,=\,{\mathfrak s}{\mathfrak o}(4)$ is
$$
e_1\,=\, e_{12}-e_{21}\, ,e_2\,=\, e_{13}-e_{31}\, ,e_3\,=\, e_{14}-e_{41}\, ,
e_4\,=\, e_{23}-e_{32}\, ,e_5\,=\, e_{24}-e_{42}\, , e_6\,=\, e_{34}-e_{43}\, .
$$
Using Algorithm 3.1, a Cartan subalgebra $C$ is generated by
$\{e_1\,, e_6\}$. The roots of $C$ are
$$
a\,:=\, (\sqrt{-1}\, , \sqrt{-1})\, ,\ b\,:=\, (\sqrt{-1}\, ,
-\sqrt{-1})\, ,\ -a\, ,\ -b\, .
$$
As $a+b$ is not a root, this root system is of type $A_1\times A_1$.

Also, conjugation maps a root to its negative. Thus the subalgebras
generated by the eigen-spaces $V_r\, , V_{-r}$, $r\,=\, a\, , b$,
contain a real form of $\text{sl}(2, {\mathbb C})$, which must be
isomorphic to ${\mathfrak s}{\mathfrak o}(3)$.

In more detail, we need to compute only the eigen-vectors for the positive roots. Their
real and imaginary parts will give the decomposition of the compact algebra $V$.
Now,
$$
V_a\,=\, \langle e_2+\sqrt{-1}e_3+\sqrt{-1}e_4-e_5\rangle\, ,\
V_b\,=\, \langle e_2+\sqrt{-1}e_3-\sqrt{-1}e_4+e_5\rangle\, .
$$
The real and imaginary parts of the basis elements in $V_a$ and $V_b$ and
generate ${\mathfrak s}{\mathfrak o}(4)$.

Let
$$
u_1\,=\, e_2-e_5\, , \ v_1\,=\, e_3+e_4\, ,\ u_2\,=\, e_2+e_5\, ,\ v_2\,=\, e_3-e_4\, .
$$
Then $u_1\, , v_1\, , [u_1\, , v_1]\,=\, -2(e_1+e_6)$ generate a copy
of ${\mathfrak s}{\mathfrak o}(3)$. Note that $u_2\, , v_2\, , [u_2\, , v_2]\,=\,
2(e_1-e_6)$ also generate a copy of ${\mathfrak s}{\mathfrak o}(3)$.
These two copies of ${\mathfrak s}{\mathfrak o}(3)$ commute because the root system
is of type $A_1\times A_1$. This gives the well known fact that ${\mathfrak s}{\mathfrak
o}(4)\,=\,{\mathfrak s}{\mathfrak o}(3)\oplus {\mathfrak s}{\mathfrak o}(3)$.

\subsection{${\mathfrak s}{\mathfrak o}(1,3)$}

A basis of $V\,=\, {\mathfrak s}{\mathfrak o}(1,3)$ is
$$
e_1\,=\, e_{12}+e_{21}\,
,e_2\,=\, e_{13}+e_{31}\, ,e_3\,=\, e_{14}+e_{41}\, , e_4\,=\, e_{23}-e_{32}\,
,e_5\,=\, e_{24}-e_{42}\, , e_6\,=\, e_{34}-e_{43}\, .
$$
Using Algorithm 3.1 a Cartan subalgebra $C\,=\, \langle e_1\, , e_6\rangle$ is obtained.
Note that there is no real split or compact Cartan subalgebra.

The roots of $C$ are
$$
a\,:=\, (1\, , -\sqrt{-1})\, ,\ b\,:=\, (1\, , \sqrt{-1})\, , \-a\, ,\
-b\, .
$$
The root system is of type $A_1\times A_1$, with positive roots $a\, , b$ and
conjugation maps $a$ to $b$.

The real rank is one, and the eigen-values of $\text{ad}(e_1)$ are $-1,\,-1,\,0,\,0,\,1,\,1$.
The corresponding root spaces are
$$
V_1\,=\, \langle e_3+e_5\, , e_2+e_4\rangle\, ,
V_{-1}\,=\, \langle -e_3+e_5\, , -e_2+e_4\rangle\, .
$$
Also, $[e_3+e_5\, , -e_3+e_5]\,=\,2e_1$; the subalgebra generated by
$e_3+e_5\, , -e_3+e_5$ is $\text{sl}(2,{\mathbb R})$, while the
subalgebra generated by $e_4\, , e_5\, , e_6$ is ${\mathfrak s}{\mathfrak o}(3)$.

Thus a maximal solvable subalgebra consisting of elements with real eigen-values in the
adjoint representation is
$$
\langle e_1\, , e_3+e_5\, , e_2+e_4\rangle\, ,
$$
and a maximal solvable subalgebra is
$$
\langle e_6\, ,\, e_1\, ,\, e_3+e_5\, ,\, e_2+e_4\rangle\, .
$$

\subsection{${\mathfrak s}{\mathfrak o}(2,2)$}

A basis of $V\,=\, {\mathfrak s}{\mathfrak o}(2,2)$ is
$$
e_1\,=\, e_{12}-e_{21}\,
,e_2\,=\, e_{13}+e_{31}\, ,e_3\,=\, e_{14}+e_{41}\, , e_4\,=\, e_{23}+e_{32}\,
,e_5\,=\, e_{24}+e_{42}\, , e_6\,=\, e_{34}-e_{43}\, .
$$
Using Algorithm 3.1, a real split Cartan subalgebra is $C\,=\,
\langle e_2\, , e_5\rangle$, while a compact Cartan subalgebra is
$\langle e_1\, , e_6\rangle$.

The roots of $C$ are
$$
a\,:=\, (1\, , 1)\, , \ b\,:=\, (1\, , -1)\, ,\ -a\, ,\ -b\, .
$$
The root spaces are
$$
V_a\,=\, \langle e_1-e_3+e_4-e_6\rangle \, ,\ V_b\,=\, \langle e_1+e_3+e_4+e_6\rangle
$$
$$
V_{-a}\,=\, \langle e_1+e_3-e_4-e_6\rangle\, ,\
V_{-b}\,=\, \langle e_1-e_3-e_4+e_6\rangle\, .
$$
Conjugation fixes the roots. Consequently, the subalgebra generated
by a root spaces of a root and its negative is isomorphic to
$\text{sl}(2,{\mathbb R})$.

Therefore, denoting the subalgebra generated by $V_r\, , V_{-r}$ by $\langle
V_r\, , V_{-r}\rangle$ the decomposition $\langle
V_a\, , V_{-a}\rangle\oplus \langle V_b\, , V_{-b}\rangle$ gives an isomorphism
of $\text{sl}(2,{\mathbb R})\oplus \text{sl}(2,{\mathbb R})$ with
${\mathfrak s}{\mathfrak o}(2,2)$.

\subsection{Lie symmetries of wave equations}\label{se4.4}

The algebra of Lie point symmetries of the wave equation in a flat 4-d space is sixteen
dimensional and determined by the vector fields (following the same order as given in
\cite{ADGM} and using the notation in which $X$ is represented with $e$):

\begin{align*}
&e_{1} = yt \partial_{t}+xy\partial_{x}+\frac{(y^2+t^2-x^2-z^2)}{2}\partial_{y}+yz\partial_{z}-uy\partial_{u}, \\
&e_{2} = y\partial_{t} + t\partial_{y}, \\
&e_{3} = xt \partial_{t}+\frac{(x^2+t^2-y^2-z^2)}{2}\partial_{x}+xy\partial_{y}+xz\partial_{z}-ux\partial_{u}, \\
&e_{4} = x\partial_{t} + t\partial_{x}, \\
&e_{5} = zt \partial_{t}+zx\partial_{x}+yz\partial_{y}+\frac{(z^2+t^2-y^2-x^2)}{2}\partial_{z}-uz\partial_{u},\\
&e_{6} = z\partial_{t} + t\partial_{z}, \\
&e_{7} = t\partial_{t} + x\partial_{x}+y\partial_{y} + z\partial_{z}, \\
&e_{8} = \partial_{t}, \\
&e_{9} = (t^2+x^2+y^2+z^2) \partial_{t} + 2tx \partial_{x}+ 2ty \partial_{y}+ 2tz \partial_{z}- 2ut\partial_{u}, \\
&e_{10} = \partial_{y}, \\
&e_{11} = \partial_{x}, \\
&e_{12} = \partial_{z}, \\
&e_{13} = z\partial_{y}-y\partial_{z}, \\
&e_{14} = z\partial_{x}-x\partial_{z}, \\
&e_{15} = y\partial_{x}-x\partial_{y}, \\
&e_{16} = u\partial_{u}.
\end{align*}

The commutator algebra of the finite dimensional part of the symmetry algebra of 
the wave equation on Minkowski space-time is 15 dimensional. By computing its 
radical or the Killing form one sees that this 15-dimensional algebra is 
semisimple.

Its basis is $e_1,\ldots,e_6,\,e_7-e_{16},e_{8},\ldots, e_{15}.$ For
notational convenience we will denote-only in this section-
$e_7-e_{16}$ by $e_7$.

The commutator table is reproduced in Appendix 1. The translations
parallel to the coordinate axes are
$$
\partial_x\,=\, e_{11}\, ,\partial_y\,=\, e_{10}\, ,\partial_z\,=\, e_{12}\, ,
\partial_t\,=\, e_8
$$
and they form an ad--nilpotent subalgebra.

Let
$$
U\, =\, \langle e_8\, , e_{10}\, , e_{11}\, ,e_{12}\rangle\, .
$$
We will use standard Maple commands and Algorithm 3.3 to embed $U$ in a maximal 
ad-nilpotent subalgebra $\widetilde{U}$ and also compute the normalizer of this 
subalgebra. It may be pointed out that the implementation of Algorithm 3.3 will give 
rise to the identification of several lower dimensional non-conjugate subalgebras. 
Since the standard symmetry reduction from translations yield trivial solutions, 
therefore our algorithm will provide translations embedded into those subalgebras 
that will provide non-trivial solutions of the wave equation. Consequently this 
approach provides a direct use of translational subalgebras.

As explained in Algorithm 3.3, in general, if
$$
N(\widetilde{U})\, =\, S+R
$$
is the Levi decomposition of $N(\widetilde{U})$ then $S$ has a
negative definite Killing form, $R'$ is contained in $U$ and $R/U$ consists only of
semisimple elements in the sense that if we extend a basis of $U$ to a basis of $R$
and find the Jordan
decomposition of the basis elements outside $U$, then the nilpotent parts all belong
to $U$.

As the abelian algebra representing $R/U$ is a torus and its real part $A$ is a
maximal abelian
algebra consisting of real semisimple elements, $A$ can be enlarged to a maximally split
Cartan algebra of the whole algebra, using Algorithm 3.1. This algebra permutes the
common eigen-spaces of $A$.

Following Algorithm 3.3, we first compute $N(U)$ and its Levi decomposition.
We have
$$
N(U)\,=\, R\oplus S\, ,
$$
where $R\,=\, \langle U\, , e_7\rangle$ is the radical, and
$$
S\,=\,
 \langle e_{15}\, , e_{14}\, , e_{13}\, ,e_{6}\, , e_{5}\, , e_{4} \rangle
$$
is semisimple with Cartan decomposition
$$
S\,=\, \langle e_{15}\, , e_{14}\, , e_{13}\rangle\oplus
\langle e_{6}\, , e_{5}\, , e_{4} \rangle
$$
with compact part $K\, =\, \langle e_{15}\, , e_{14}\, ,
e_{13}\rangle$ and radial part $P\,=\, \langle e_{6}\, , e_{5}\, ,
e_{4} \rangle$.

The element $e_7$ representing $R/U$ is real semisimple and $e_6$ is maximal abelian
in $P$. The compact subalgebra $K$ is the subalgebra of spatial rotations.

The eigen-values of $\text{ad}(e_6)$ in $S$, counting multiplicities,
are $1,1,-1,-1,0,0$ and eigen-vectors for eigen-value 1 are
$-e_{15}+e_4$, $-e_{13}+e_6$. Therefore, as the eigen-vectors for
positive eigen-values of a real semisimple element of $S$ form an
ad-nilpotent subalgebra, following Algorithm 3.3, we adjoin
$-e_{15}+e_4$, $-e_{13}+e_6$ to $U$ to get an ad-nilpotent algebra
$\tilde{U}$ and compute its normalizer. We find that
$$
N(\widetilde{U})\,=\, \langle \widetilde{U}\, , e_2\, ,e_7\, ,e_{14}\rangle\, .
$$
The subalgebra $\langle e_2\, ,e_7\, ,e_{14}\rangle$ is abelian, and is a torus, whose
real part is $\langle e_2\, ,e_7\rangle$ and compact part is $\langle e_{14}\rangle$.

Thus $N(\widetilde{U})$ is self--normalizing and solvable. Therefore, by Algorithm 3.3,
$\widetilde{U}$ is a maximal ad-nilpotent subalgebra containing $U$.
Using Algorithm 3.1, we find that $\langle e_2\, ,e_7\, ,e_{14}\rangle$ is a Cartan
subalgebra and $A\,=\,\langle e_2\, ,e_7\rangle$ is a maximal abelian
subalgebra of real semisimple elements.

The roots of $A$ on $N(\widetilde{U})$ are
$$
(-1\, ,0)\, , (-1\, ,-1)\, , (-1\, ,1)\, , (0\, ,1)\, ;
$$
here, to say that $(r\, ,s)$ is a root means that there is a common
eigen-vector $X$ for $A$ which is not centralized by $A$ and
$$
[e_7\, , X]\, =\, rX\, , [e_2\, , X]\, =\, sX\, .
$$
Let
$$
a\,=\, (-1\, ,0)\, , b\,=\, (-1\, ,-1)\, , c\,=\, (-1\, ,1)\, , d\,=\, (0\, ,1)\, .
$$
This is a positive system of roots for $A$ determined by $N(\widetilde{U})$.
The only positive roots which are sums of positive roots are $a+d\,=\, c$ and $b+d\,=\, a$.
Therefore, the simple roots are $b\, , d$ and the roots as nonnegative integral
combinations of the simple roots are $b\, , d\, , b+d\, ,b+2d$.

Therefore, the Dynkin diagram of the reduced root system is of type
$B_2$ with $b$ a long root.

Let $\omega_7\, , \omega_2$ be linear functions on $A$ dual to the
ordered basis $e_7\, , e_2$. With this notation, the roots are
$$
-\omega_7\, ,-\omega_7-\omega_2\, ,-\omega_7+\omega_2\, , \omega_2\,
.
$$
Let $L\,=\,N(\widetilde{U})$. The corresponding eigen-spaces in $L$ are
$$
L_{-\omega_7}\,=\, \langle e_{12}\, ,e_{11}\rangle\,
,L_{-\omega_7-\omega_2}\,=\,\langle e_8+e_{10} \rangle\, ,
L_{-\omega_7+\omega_2}\,=\, \langle e_8-e_{10}\rangle\, ,
L_{\omega_2}\,=\, \langle -e_{13}+e_6\, ,-e_{15}+e_4\rangle\, .
$$
Finally $L_0\,=\, \langle A\, ,e_{14}\rangle$ and $e_{14}$ operates
on these eigen-spaces, as rotations on $L_{-\omega_7}$ and
$L_{\omega_2}$, while it commutes with $L_{-w_7-w_2}$ and
$L_{-\omega_7+\omega_2}$.

The absolute root system is determined by common eigen-vectors for the
Cartan algebra
$$
C\,=\, \langle e_7\, ,e_2\, ,e_{14}\rangle\, .
$$
The positive roots are
$$
a\,=\, (0,1,-\sqrt{-1})\, , b\,=\,(0,1,\sqrt{-1})\, , c\,=\, (1,0,\sqrt{-1})
$$
$$
d\,=\, (1,1,0)\, , e\,=\, (1,-1,0)\, ,f\,=\, (1,0,-1)\, .
$$
Forming sums of pairs of positive roots and removing those roots
that are sums of positive roots, we find that the simple roots are
$a\, ,e\, ,b$ with Dynkin diagram of type $A_3$ with $e$ the simple
middle root.

Conjugation maps $a$ to $b$ and fixes $e$. Thus the algebra is a
real form of $\text{sl}(4,{\mathbb C})$.

To find a maximally compact subalgebra -- if any -- we follow a
procedure analogous to Algorithm 3.1 -- starting with a compact
element -- namely one which generates a compact subgroup. For
example, as $\langle e_{15}\, ,e_{14}\, ,e_{13}\rangle$ generate
${\mathfrak s}{\mathfrak o}(3)$, because $e_{15}\,=\,
x\partial_y-y\partial_x$ and $e_{14}\,=\, x\partial_z-z\partial_x$,
we can with begin with $e_{15}$, compute its centralizer, the center
of its centralizer and its derived algebra. If the derived algebra
is trivial, then the centralizer $e_{15}$ of would be a maximal
torus and its compact part will be a maximal compact subalgebra
containing $e_{15}$. If the derived algebra is nontrivial, it must
have a compact element, say $t$. Adjoining it to $e_{15}$ and
computing the centralizer of $\langle e_{15}\, ,t\rangle$ and its
derived algebra and repeating the process, we will ultimately obtain
a maximal compact subalgebra containing $e_{15}$.

In this case we find that $C_k\,=\, \langle t_1\, ,t_2\, ,e_{15}\rangle$ is a
maximally compact Cartan subalgebra, where
$$
t_1\,=\, 2e_{12}+e_5\, ,t_2\,=\, e_9+4e_8\, .
$$
The positive roots are
$$
a\,=\, (\sqrt{-1}, -4\sqrt{-1},0)\, , b\,=\,(0,4\sqrt{-1},\sqrt{-1})\, , c\,=\,
(\sqrt{-1},0,-\sqrt{-1})\, ,
$$
$$
d\,=\, (0,4\sqrt{-1},-\sqrt{-1})\, , e\,=\, (\sqrt{-1},4\sqrt{-1},0)\, ,f\,=\,
(\sqrt{-1},0,\sqrt{-1})\, .
$$
The simple positive roots are $d\, ,a\, ,b$ with Dynkin diagram of
type $A_3$ with $a$ the middle simple root.

Conjugation maps every root to its negative. The root algebras generated by the real
and imaginary parts of the root vectors are copies of $\text{sl}(2,{\mathbb R})$ except
for roots $d+a\, ,a+b$, where they
generate copies of ${\mathfrak s}{\mathfrak o}(3)$.
Specifically, the subalgebras generated by the real and imaginary parts of root vectors
for $d+a$ and $a+b$ are
$$
\langle e_1+2e_{10}-2e_{14},\, e_3+2e_{11}+2e_{13},\, -4e_5- 8e_{12}+8e_{15}\rangle\, ,
$$
$$
\langle e_1+2e_{10}+2e_{14},\,-e_3-2e_{11}+2e_{13},\, -4e_5- 8e_{12}-8e_{15}\rangle\, .
$$
Both are isomorphic to ${\mathfrak s}{\mathfrak o}(3)$. Denoting
these subalgebras by $k_1$ and $k_2$ respectively, we find that the
centralizer of $k_1$ is
$$
k_2\oplus\langle 4e_8+e_9\rangle\, .
$$
Moreover as the centralizer of the copy of ${\mathfrak s}{\mathfrak o}(3)$ given by
$k_0\,=\,\langle e_{15}\,
,e_{14}\, ,e_{13}\rangle$ is $\langle e_7\,
,e_8\, ,e_9\rangle\,=\,\text{sl}(2,{\mathbb R})$, the subalgebra $k_1$ is not conjugate
to $k_0$.

Finally, as the Killing form of the full algebra has seven negative eigen-values, a maximal
compact subalgebra is
$$
k_1\oplus k_2\oplus \langle 4e_8 +e_9\rangle
$$
because $4e_8 +e_9$ generates the maximal compact subalgebra of
$\langle e_7\, ,e_8\, ,e_9\rangle\,=\,\text{sl}(2,{\mathbb R})$.

\section{Lie Symmetries of $f_{xx}=\frac{4}{3}f_{yy}^3, f_{xy}=f^2_{yy}$ and
$v'= (u'')^2$}\label{sec5}

These equations were considered by Cartan in the context of symmetries of a certain 
system of equations defined by differential forms \cite{Ca};
he showed that their symmetry algebra was the 14 dimensional simple group $G_{2}$.
Maple is able to compute both the
algebras by using commands for contact symmetries and for
generalized symmetries as well as its root space decomposition and
its maximal compact subalgebra. The latter equation was also
considered by Anderson, Kamran and Olver, \cite{AKO}, in the context
of generalized symmetries. To illustrate the algorithms of this
paper, we will use the table given in \cite{AKO} -- reproduced in
Appendix 2 to identify the algebra and determine several interesting
subalgebras. To streamline the calculations we will use repeatedly
the following facts-already mentioned in the Introduction.

If $H$ is a semisimple subalgebra of a semisimple algebra $G$ then
and element $X$ of $H$ is real semisimple, compact or nilpotent in
the adjoint representation $H$ of on itself,if and only if it is,
respectively, real semisimple, compact or nilpotent in the adjoint
representation of $H$ in $G$. Moreover, the derived algebras of the
radical of normalizer or centralizer of any subalgebra are
nilpotent.

The symmetry algebra has basis $X_1, X_2, \cdots , X_{14}$ and is given by:

\begin{align*}
&X_{1} = (\frac{2}{3}u'^2-uu'')\partial_{x}+(\frac{1}{2}uv+\frac{4}{9}u'^3-uu'u'')\partial_{u}+(\frac{1}{2}v^2-\frac{1}{3}uu''^3)\partial_{v}, \\
&X_{2}=(\frac{4}{3}x^2u'-2xu-\frac{1}{3}x^3u'')\partial_{x}+(\frac{1}{6}x^3v+\frac{2}{3}x^2u'^2-2u^2-\frac{1}{3}x^3u'u'')\partial_{u}+(2xu'v-2uv-\frac{1}{9}x^3u''^3-\frac{8}{9}u'^3)\partial_{v}, \\
&X_{3} = (\frac{8}{3}xu'-2u-x^2u'')\partial_{x}+(\frac{1}{2}x^2v+\frac{4}{3}xu'^2-x^2u'u'')\partial_{u}+(2vu'-\frac{1}{3}x^2u''^3)\partial_{v}, \\
&X_{4} = (\frac{8}{3}u'-2xu'')\partial_{x}+(xv+\frac{4}{3}u'^2-2xu'u'')\partial_{u}-\frac{2}{3}xu''^3\partial_{v}, \\
 \\
&X_{5} = -2u''\partial_{x}+(v-2u'u'')\partial_{u}-\frac{2}{3}u''^3\partial_{v}, \\
&X_{6} = \frac{1}{2}u\partial_{u} + v\partial_{v}, \\
&X_{7} = -\frac{1}{2}x^2\partial_{x} -\frac{3}{2}xu\partial_{u}-2u'^2\partial_{v}, \\
&X_{8} = -x\partial_{x}-\frac{3}{2}u\partial_{u}, \\
&X_{9} = -\partial_{x} , \\
&X_{10} = \frac{1}{6}x^3\partial_{u}+2(xu'-u)\partial_{v}, \\
&X_{11} = \frac{1}{2}x^2\partial_{u}+2u'\partial_{v}, \\
&X_{12} = x\partial_{u}, \\
&X_{13} = \partial_{u}, \\
&X_{14} = \partial_{v}.
\end{align*}

Using the commutator table in Appendix 2, and computing the
determinant of the Killing form by Maple, we find that it is
non-zero. Thus, the algebra is semisimple. The translations
$$
X_{14}\,=\, \partial_v\, ~~~ X_{13}\,=\, \partial_u
$$
clearly commute and $X_{12}\,=\, x\partial_u$ commutes with both.
One can check that they are nilpotent: this also follows from
computing the derived algebra of the radical of normalizer of U.
We want to embed
$$
U\,=\, \langle X_{14}\, , X_{13}\, , X_{12}\rangle
$$
in a maximal subalgebra whose
elements are all nilpotent. Following Algorithm 3.3 we compute the normalizer $N(U)$ and find
its Levi decomposition, using standard Maple commands:
$$
N(U)\,=\, \langle X_9\, , X_8\, , X_6\, , X_5\, , X_{14}\, , X_{13}\, , X_{12}
\, , X_{11}\, , X_{10}\rangle
$$
and its Levi decomposition is $R(N(U))\oplus S$, where the radical
$$
R(N(U))\,=\, \langle X_9\, , X_8 - 3X_6\, , X_{14}\, , X_{13}\, , X_{12}\, , X_{11}\rangle
$$
and the semisimple part $S\,=\, \langle X_8 + X_6\, , X_5\, , X_{10}\rangle$.
The commutators for the semisimple part are
$$
[X_8 + X_6\, , X_5] \,=\, 2X_5,\, [X_8 + X_6\, , X_{10} ]\,=\,
-2X_{10},\, [X_5 ,\, X_{10}]\,= -2(X_8 + X_6)\, .
$$
This means that $X_5$ and $X_{10}$ are nilpotent in the full algebra, $X_8 + X_6$ is
real semisimple in the full algebra and $X_5 + X_{10}$ is a compact element.
Following Algorithm 3.3, we compute the derived algebra of the radical $R(N(U))$.
It is
$$
{\widetilde U}\,=\,\langle X_9\, , X_{14}\, , X_{13}\, , X_{12}\, , X_{11}\rangle\, .
$$
The quotient $R(N(U)) /{\widetilde U}$ is represented by $X_8 - 3X_6$, which is a real
semisimple
element. (This also follows by computing the centralizer of $X_5 + X_{10}$ and its derived
algebra, which turns out to be $\langle X_8 -3X_6\, , X_{12}\, , X_3\rangle$.
This is a standard ${\rm sl}(2, {\mathbb R})$ with $X_8 -3X_6$ as real semisimple element.)

Following Algorithm 3.3 we compute again $N(\widetilde{U})$ and its Levi decomposition.
It turns out to be identical to the Levi decomposition of $N({U})$. We therefore
adjoint a nilpotent element coming from the semisimple part of the decomposition, say
$X_5$. Let
$$ \widetilde{\widetilde{U}}\,=\, \langle \widetilde{U}, X_5\rangle\, .
$$
Its normalizer is $\langle \widetilde{\widetilde{U}}\, , X_8\, , X_6\rangle$ and it is
solvable, with commutator $\widetilde{\widetilde{U}}$ and the quotient is represented
by the real torus $\langle X_6\, , X_8\rangle$. This also follows from noticing that
$X_8 + X_6$ is also real semisimple and commutes with $X_8 - 3X_6$.

Thus a maximal nilpotent subalgebra containing
$$
U\,=\, \langle X_{14}\, , X_{13}\, , X_{12}\rangle
$$
is $\widetilde{\widetilde{U}}\,=\,\langle X_5\, , X_{14}\, , X_{13}\, , X_{12}\, ,
X_{11}\, , X_9\rangle$. Finally, $C \,=\,\langle X_6\, , X_8\rangle$ is self-centralizing
and it is a
real split Cartan subalgebra of the full 14 dimensional algebra $L$.

Maple gives the following roots for $C$ in $\widetilde{\widetilde{U}}$: in fact, the
basis vectors for $\widetilde{\widetilde{U}}$ listed above are common eigen-vectors
for $C$ with eigen-values
$$
a\,=\,(\frac{1}{2}, \frac{3}{2}),\, b\,=\, (-1,0),\, c\,=\, (-\frac{1}{2}, \frac{3}{2}),\,
d\,=\, (-\frac{1}{2}, \frac{1}{2}),\, e\,=\, (-\frac{1}{2}, -\frac{1}{2}),\, f\,=\,
(0,1)\, .
$$
As explained in Section \ref{sR}, this is a positive system of roots and a simple
system of roots is given by adding pairs of positive roots and removing those that are
a sum of positive roots. We have
$$
a + b\,=\, c,\, a + e\,=\, f,\, d + e\,=\, b,\, e + f\,=\, d\, .
$$
Thus the simple roots are $a, e$ and the positive roots written in terms of these
roots are
$$
a,\, e,\, a + e \,=\, f,\, a + 2e \,= \, d,\, a + 3e \,=\, b,\, 2a + 3e \,=\, c\, .
$$
Therefore, the algebra $L$ is of type $G_2$ with a real split Cartan subalgebra. Any
semisimple split real Lie algebra is generated by copies of
${\rm sl}(2,{\mathbb R})$ corresponding to the
simple roots, with relations
$$
[X\, , Y] \,=\, H,\, [H\, , X] \,=\, 2X,\, [H\, , Y] \,=\, -2Y
$$
and its maximal compact subalgebra is generated by copies
of the compact element $X-Y$, which generates a circle, in these generating root
${\rm sl}(2,{\mathbb R})$ copies; see \cite[pp.~99--100]{St} for a global version of
these results.

Here, the root vectors corresponding to $a, -a$ are $X_5 ,X_{10}$;
the root vectors corresponding to $e, -e$ are $X_{11} , X_4$ and a maximal compact
compact subalgebra $K$ is thus generated by
$$
J_1 \,=\, X_5 + X_{10} ,\, J_2 \,=\, X_4 - X_{11}\, .
$$
The algebra is spanned by
$$
J_1 ,\, J_2 ,\, J_3 \,=\, X_1 + \frac{3}{8}X_{14} ,\, J_4 \,=\, X_2 - \frac{3}{4}X_{13} ,\,
 J_5 \,=\, X_3 +\frac{3}{4}X_{12} ,\, J_6 \,=\, X_7 - \frac{3}{2}X_9\, .
$$
To identify the structure of $K$, we must choose a Cartan subalgebra of $K$ and compute its
roots in the complexification of $K$ --- exactly as for ${\mathfrak s}{\mathfrak o}(4)$ in
Section \ref{sec4.1l}.
Now the centralizer of $J_1$ is $\langle J_1\, , J_5\rangle$, and it is therefore a
Cartan subalgebra of $K$. Its positive roots are
$$
(\sqrt{-2}\, , -\frac{1}{\sqrt{-2}}),\, (\sqrt{-2}\, , \frac{3}{\sqrt{-2}})\, .
$$
Therefore the system is of type $A_1\times A_1$.
The real and imaginary parts for the root vectors of $(\sqrt{-2}\, , -\frac{1}{\sqrt{-2}})$
are
\begin{equation}\label{g1}
J_3 - \frac{1}{6} J_6 ,\, \frac{\sqrt{2}}{4}J_4- \frac{\sqrt{2}}{8}J_2
\end{equation}
and for the root $(\sqrt{-2}\, , \frac{3}{\sqrt{-2}})$ they are
\begin{equation}\label{g2}
J_3 + \frac{1}{2} J_6 ,\, \frac{-\sqrt{2}}{4}J_4+
\frac{3\sqrt{2}}{8}J_2\, .
\end{equation}
The vectors in \eqref{g1} generate a copy of ${\mathfrak s}{\mathfrak o}(3)$ and
in \eqref{g2} also a copy of ${\mathfrak s}{\mathfrak o}(3)$ and these subalgebras commute.

This gives an explicit decomposition of a maximal compact subalgebra of $L$ as a sum of
copies of ${\mathfrak s}{\mathfrak o}(3)$.

\section{Solutions of the wave equation}

Section \ref{Se4} gives several non-conjugate subalgebras of the symmetry algebra of the 
wave equation. The reason that they are non-conjugate is that the structure of low 
dimensional Lie algebras is well documented in literature \cite{PBNL}, \cite{SW}.
In the case of three dimensional algebras, this is the 
Lie-Bianchi classification \cite[p.~479--562]{Lie1}, \cite{Bian}
In this case, this can 
be described very briefly. If $L$ is a 3 dimensional algebra and its
commutator $L'\,=\,[L,\, L]$ is 1--dimensional, then $L$ is completely determined by
the dimension of the centralizer of $L'$ in $L$; if $L'$ is 2--dimensional, then
$L'$ is abelian and the 
structure of $L$ is completely determined by the eigen-values of $L/L'$ in $L'$ and their 
multiplicities; in case $L'$ is of dimension 3, the eigen-values of a single element 
suffice to determine the structure of $L$ \cite[Corollary~2.2~and~Section 4.3]{ADMM}.
The algebras given below were of all possible Lie--Bianchi types. Their
identification is facilitated by determining the reduced root system using the 
algorithms of Section \ref{aL}, by enlarging a given subalgebra of commuting ad-nilpotent 
elements to a maximal solvable subalgebra that contains up to conjugacy all 
solvable subalgebras with real eigen-values. If the vector fields are in 
polynomial form and contain translations with respect to the independent variables, then 
these translations are ad-nilpotent. In case the Cartan algebra so obtained has a 
compact part, it must operate on the positive root spaces of its real part and this 
way one may obtain all 3 dimensional subalgebras of solvable subalgebras of all
Lie--Bianchi types.

In the specific example of the 15
dimensional algebra considered in Section \ref{se4.4}, we 
denoted --- for simplicity of notation the element $e_7 -e_{16}$ by $e_7$. Taking
this into account, the Cartan algebras obtained in Section \ref{se4.4} were
$\langle e_2,\, e_7 -e_{16},\, e_{14}\rangle$ and $\langle 2e_{12} + e_5,
\,e_9+ 4e_8,\, e_{15} \rangle$. This first Cartan 
algebra is maximally real with its real part 
$A\,=\, \langle e_2,\, e_7 -e_{16}\rangle$. For this reason, 
the relative root system is different from the absolute root system. The roots of 
$A$ were determined in Section \ref{se4.4} as
$$
b\,=\, -\omega_7-\omega_2\, , \ d\,=\, \omega_2\, , \ b+d\,=\, -\omega_7\, , 
\ b+2d\,=\, \omega_2- \omega_7\, ,
$$
where $\omega_2,\, \omega_7$ are dual to the ordered basis $e_2,\, e_7 -e_{16}$
of $A$.

The root spaces of $A$ in the maximal solvable algebra $L$ determined in
Section \ref{se4.4} were 
\begin{itemize}
\item $L_b\,=\, L_{-\omega_7-\omega_2}\,=\, \langle e_8,\, e_{10}\rangle$,

\item $L_d\,=\, L_{\omega_2}\,=\, \langle e_6-e_{13},\, e_{4}-e_{15}\rangle$,

\item $L_{b+d}\,=\, L_{-\omega_7}\,=\, \langle e_{12},\, e_{11}\rangle$,

\item $L_{b+2d}\,=\, L_{\omega_2-\omega_7}\,=\, \langle e_{8}- e_{10}\rangle$.
\end{itemize}
Thus, this displays the common eigen-vectors of $A$ and 
their multiplicities. Moreover as $[L_r,\, L_s]\, \subset\, L_{r+s}$,
and $2r$ is not a root, the 
root spaces given above contain commuting eigen-vectors of $e_2$ with different 
eigen-values and of $e_7 -e_{16}$ with repeated eigen-values. As any element centralizing 
$A$ operates on each root space of $A$, applying this to the compact part of the Cartan 
algebra $\langle A,\, e_{14}\rangle$ gives all possible solvable
3--dimensional solvable Lie--Bianchi types. Finally, using
the compact Cartan subalgebra, we found a maximal compact subalgebra,
whose derived algebra was ${\rm so}(4)\,=\, {\rm so}(3)\oplus {\rm so}(3)$. The
centralizer for the spatial rotations was ${\rm sl}(2,{\mathbb R})$. For this
reason, the 3--dimensional simple algebras given below are 
non-conjugate. We now proceed to find the corresponding reductions and invariant 
solutions.

A preliminary step is to find the invariants of a given algebra of vector
fields. The number of functionally independent invariants can be found
from the row reduced echelon form of the operators. In the row reduced
echelon form, the resulting operators always commute \cite{ABGM}.

The equation is
\begin{equation}\label{eq-w}
{\frac {\partial ^{2}}{\partial {x}^{2}}}u \left( t,x,y,z
 \right) +{\frac {\partial ^{2}}{\partial {y}^{2}}}u \left( t,x
,y,z \right) +{\frac {\partial ^{2}}{\partial {z}^{2}}}u \left(
t,x,y,z \right) -{\frac {\partial ^{2}}{\partial {t}^{2}}}u
 \left( t,x,y,z \right) =0
\end{equation}

\noindent \textbf{(I) }{$\dim G' =0$}

\noindent \textbf{(I.a) } ${\mathcal L}_{1,0}
\,=\, \langle e_{8},e_{10},e_{11}\rangle$

The joint invariants of ${\mathcal L}_{1,0}$ are
\begin{equation}\nonumber
 z, u \\
\end{equation}
so that the corresponding similarity transformations
\begin{equation}
 p=z, w(p)=u \\
\end{equation}
transform wave equation \eqref{eq-w} to
\begin{equation}
 {\frac {{\rm d}^{2}}{{\rm d}{p}^{2}}}w \left( p \right) =0
\end{equation}
which has the solution
\begin{equation}
w \left( p \right) ={\it C_1}\,p+{\it C_2}.
\end{equation}
This leads to solution
\begin{equation}
u \left( t,x,y,z \right) ={\it C_1}\,z+{\it C_2}
\end{equation}
of wave equation \eqref{eq-w}. \\

\noindent \textbf{(I.b) }${\mathcal L}_{2,0} \,=\,\langle e_{2},e_7 -
e_{16},e_{14}\rangle$

The joint invariants of ${\mathcal L}_{2,0}$ are
\begin{equation}\nonumber
 -{\frac {{x}^{2}+{z}^{2}}{{y}^{2}-{{\it t}}^{2}}}, u_{{}}\sqrt {-{y}^{2}+{{\it t}}^{2}} \\
\end{equation}
so that the corresponding similarity transformations
\begin{equation}
 p=-{\frac {{x}^{2}+{z}^{2}}{{y}^{2}-{{\it t}}^{2}}}, w(p)=u_{{}}\sqrt {-{y}^{2}+{{\it t}}^{2}} \\
\end{equation}
transform wave equation \eqref{eq-w} to Jacobi ODE
\begin{equation}
 4\, \left( {\frac {{\rm d}^{2}}{{\rm d}{p}^{2}}}w \left( p \right)
 \right) {p}^{2}-4\, \left( {\frac {{\rm d}^{2}}{{\rm d}{p}^{2}}}w
 \left( p \right) \right) p+8\, \left( {\frac {\rm d}{{\rm d}p}}w
 \left( p \right) \right) p-4\,{\frac {\rm d}{{\rm d}p}}w \left( p
 \right) +w \left( p \right) =0
\end{equation}
which has the solution
\begin{equation}
w \left( p \right) ={\it C_1}\,{\it EllipticK} \left( \sqrt {p}
 \right) +{\it C_2}\,{\it EllipticCK} \left( \sqrt {p} \right)
\end{equation}
in terms of complete and complementary complete elliptic integrals
of the first kind \\
(ref:
http://www.maplesoft.com/support/help/Maple/view.aspx?path=EllipticF).
This leads to solution
\begin{equation}
u \left( {\it t},x,y,z \right) ={\frac {1}{\sqrt {-{y}^{2}+{{\it t}}
^{2}}} \left( {\it C_1}\,{\it EllipticK} \left( \sqrt {{\frac
{-{x}^{ 2}-{z}^{2}}{{y}^{2}-{{\it t}}^{2}}}} \right) +{\it
C_2}\,{\it EllipticCK} \left( \sqrt {{\frac
{-{x}^{2}-{z}^{2}}{{y}^{2}-{{\it t}} ^{2}}}} \right)\right) }
\end{equation}
of wave equation \eqref{eq-w}. \\

\noindent \textbf{(I.c) }${\mathcal L}_{3,0} \,=\, \langle e_{12} +
\frac{1}{2}e_5, e_9 + 4 e_{8},e_{15}\rangle$

The joint invariants of ${\mathcal L}_{3,0}$ are
\begin{equation}\nonumber
 {\frac {-{{\it t}}^{4}+ \left( 2\,{x}^{2}+2\,{y}^{2}+2\,{z}^{2}-8
 \right) {{\it t}}^{2}-{x}^{4}+ \left( -2\,{y}^{2}-2\,{z}^{2}
 \right) {x}^{2}-{y}^{4}-2\,{y}^{2}{z}^{2}- \left( {z}^{2}+4 \right) ^
{2}}{4\,{x}^{2}+4\,{y}^{2}}} , u \sqrt {{x}^{2}+{y}^{2}}
\end{equation}
so that the corresponding similarity transformations
\begin{equation}\label{p-3-0}
 p= {\frac {-{{\it t}}^{4}+ \left( 2\,{x}^{2}+2\,{y}^{2}+2\,{z}^{2}-8
 \right) {{\it t}}^{2}-{x}^{4}+ \left( -2\,{y}^{2}-2\,{z}^{2}
 \right) {x}^{2}-{y}^{4}-2\,{y}^{2}{z}^{2}- \left( {z}^{2}+4 \right) ^
{2}}{4\,{x}^{2}+4\,{y}^{2}}},
\end{equation}
\begin{equation}
w(p)=u \sqrt {{x}^{2}+{y}^{2}}
\end{equation}
transform wave equation \eqref{eq-w} to
\begin{equation}
 4\, \left( {\frac {{\rm d}^{2}}{{\rm d}{p}^{2}}}w \left( p \right)
 \right) {p}^{2}+8\, \left( {\frac {\rm d}{{\rm d}p}}w \left( p
 \right) \right) p-16\,{\frac {{\rm d}^{2}}{{\rm d}{p}^{2}}}w \left(
p \right) +w \left( p \right) =0
\end{equation}
which has the solution
\begin{equation}
w \left( p \right) ={\it C_1}\,{\it LegendreP} \left( -1/2,p/2
 \right) +{\it C_2}\,{\it LegendreQ} \left( -1/2,p/2 \right)
\end{equation}
in terms of Legendre functions of the first and second kind\\
 (ref: http://www.maplesoft.com/support/help/Maple/view.aspx?path=Legendre). This leads to solution
\begin{equation}
u \left( {\it t},x,y,z \right) = \frac{1}{\sqrt {{x}^{2}+{y}^{2}}}
\left( {\it C_1}\,{\it LegendreP} \left( -1/2,p/2
 \right) +{\it C_2}\,{\it LegendreQ} \left( -1/2,p/2 \right)
\right)
\end{equation}
of wave equation \eqref{eq-w} where $p$ is given by \eqref{p-3-0}. \\

\noindent \textbf{(II) }{$\dim G' =1$}

\noindent \textbf{(II.a) } ${\mathcal L}_{1,1}
\,=\,\langle e_{2},e_{7}-e_{16},e_8 + e_{10}\rangle$

The joint invariants of ${\mathcal L}_{1,1}$ are
\begin{equation}\nonumber
 {\frac {z}{x}}, u_{{}}x \\
\end{equation}
so that the corresponding similarity transformations
\begin{equation}
 p={\frac {z}{x}}, w(p)=u_{{}}x \\
\end{equation}
transform wave equation \eqref{eq-w} to
\begin{equation}
 \left( {\frac {{\rm d}^{2}}{{\rm d}{p}^{2}}}w \left( p \right)
 \right) {p}^{2}+4\, \left( {\frac {\rm d}{{\rm d}p}}w \left( p
 \right) \right) p+2\,w \left( p \right) +{\frac {{\rm d}^{2}}{
{\rm d}{p}^{2}}}w \left( p \right) =0
\end{equation}
which has the solution
\begin{equation}
w \left( p \right) ={\frac {{\it C_1}\,p+{\it C_2}}{{p}^{2}+1}}.
\end{equation}
This leads to solution
\begin{equation}
u \left( {\it t},x,y,z \right) ={\frac {{\it C_1}\,z+{\it C_2}\,x}{
{x}^{2}+{z}^{2}}}
\end{equation}
of wave equation \eqref{eq-w}. \\

\noindent \textbf{(II.b) } ${\mathcal L}_{2,1}
\,=\,\langle e_{12},-e_{6}+e_{13},-e_8 + e_{10}\rangle$

The joint invariants of ${\mathcal L}_{2,1}$ are
\begin{equation}\nonumber
x, t+y, u
\end{equation}
which gives the similarity transformation
\begin{equation}
p\,=\,x,\, q\,=\,t+y,\, w(p,q)\,=\,u_{{}}{ } \\
\end{equation}
that transforms the wave equation into
\begin{equation}
 w_{pp}\,=\,0\, ,
\end{equation}
which gives the solution
\begin{equation}
u(t,x,y,z) \,=\, xF_{1} (y+t) + F_{2} (y+t)\, .
\end{equation}

\noindent \textbf{(III) }{$\dim G' =2$}

\noindent \textbf{(III.a) } ${\mathcal L}_{1,2}
\,=\,\langle e_{7}-e_{16},e_{11}, e_{12}\rangle$

The joint invariants of ${\mathcal L}_{1,2}$ are
\begin{equation}\nonumber
 {\frac {y}{{\it t}}}, u_{{}}{\it t} \\
\end{equation}
so that the corresponding similarity transformations
\begin{equation}
 p={\frac {y}{{\it t}}}, w(p)=u_{{}}{\it t} \\
\end{equation}
transform wave equation \eqref{eq-w} to
\begin{equation}
 \left( {\frac {{\rm d}^{2}}{{\rm d}{p}^{2}}}w \left( p \right)
 \right) {p}^{2}+4\, \left( {\frac {\rm d}{{\rm d}p}}w \left( p
 \right) \right) p+2\,w \left( p \right) -{\frac {{\rm d}^{2}}{
{\rm d}{p}^{2}}}w \left( p \right) =0
\end{equation}
which has the solution
\begin{equation}
w \left( p \right) ={\frac {{\it C_1}\,p+{\it C_2}}{{p}^{2}-1}}.
\end{equation}
This leads to solution
\begin{equation}
u \left( {\it t},x,y,z \right) ={\frac {{\it C_1}\,y+{\it C_2}\,{
\it t}}{{y}^{2}-{{\it t}}^{2}}}
\end{equation}
of wave equation \eqref{eq-w}. \\

\noindent \textbf{(III.b) }\ ${\mathcal L}_{2,2}
\,=\,\langle e_{2},\, e_8+e_{10}, e_8-e_{10}\rangle$
$$
\,=\, \langle y\frac{\partial}{\partial t}+
t\frac{\partial}{\partial y},\, \frac{\partial}{\partial t}+
\frac{\partial}{\partial y},\, \frac{\partial}{\partial t}-
\frac{\partial}{\partial y}\rangle\, .$$
Clearly, joint invariants the same as the invariants of
$\langle \frac{\partial}{\partial y},\,
\frac{\partial}{\partial t}\rangle$, therefore the basic invariants are $x,\, z,\, u$.
Substituting $u\,=\, u(x,\, z)$ in the wave equation shows that $u$ must be a
harmonic function, so they are real parts of holomorphic functions in
the variable $x+\sqrt{-1}z$.\\

\noindent \textbf{(III.c) } ${\mathcal L}_{3,2} \,=\,\langle e_{14},e_{11},
e_{12}\rangle$.

The joint invariants of ${\mathcal L}_{3,2}$ are
\begin{equation}\nonumber
 t,\, y, \, u \\
\end{equation}
so that the corresponding similarity transformations
\begin{equation}
 p\,=\,t,\, q\,=\,y, \, w(p,q)\,=\,u \\
\end{equation}
transform wave equation \eqref{eq-w} to
\begin{equation}
 {\frac {\partial^{2}}{\partial {q}^{2}}}w \left( p,q \right) -{\frac
{\partial ^{2}}{\partial {p}^{2}}}w \left( p,q \right) =0
\end{equation}
which has the solution
\begin{equation}
w \left( p,q \right) ={\it F_1} \left( q+p \right) +{\it F_2}
 \left( q-p \right).
\end{equation}
This leads to solution
\begin{equation}
u \left( {\it t},x,y,z \right) ={\it F_1} \left( y+{\it t} \right)
+{\it F_2} \left( y-{\it t} \right)
\end{equation}
of wave equation \eqref{eq-w}. \\

\noindent \textbf{(III.d) } ${\mathcal L}_{4,2} \,=\,\langle e_{14},- e_6+
e_{13}, -e_4 + e_{15}\rangle$

The joint invariants of ${\mathcal L}_{4,2}$ are
\begin{equation}\nonumber
 y+{\it t}, {x}^{2}-2\,{\it t}\,y+{z}^{2}-2\,{{\it t}}^{2}, u \\
\end{equation}
so that the corresponding similarity transformations
\begin{equation}
 p\,=\,y+{\it t},\ q\,=\, {x}^{2}-2\,{\it t}\,\ y+{z}^{2}-2\,{{\it t}}^{2},\ w(p,q)\,=\,u \\
\end{equation}
transform wave equation \eqref{eq-w} to
\begin{equation}
\left( -{p}^{2}+q \right) {\frac {\partial ^{2}}{\partial {q}^{2}}}w
 \left( p,q \right) + \left( {\frac {\partial ^{2}}{\partial q
\partial p}}w \left( p,q \right)\right) p+2\,{\frac {\partial }{
\partial q}}w \left( p,q \right) =0
\end{equation}
which has the solution
\begin{equation}
w \left( p,q \right) ={\frac {1}{p} \left( {\it F_2} \left( p
 \right) p+{\it F_1} \left( {\frac {{p}^{2}+q}{p}} \right) \right) }.
\end{equation}
This leads to solution
\begin{equation}
u \left( {\it t},x,y,z \right) ={\it F_2} \left( y+{\it t} \right)
+{\frac {1}{y+{\it t}}{\it F_1} \left( {\frac {{x}^{2}+{y}^{2}+{z}^{
2}-{{\it t}}^{2}}{y+{\it t}}} \right) }
\end{equation}
of wave equation \eqref{eq-w}. \\

\noindent \textbf{(IV) }{$\dim G' =3$}

\noindent \textbf{(IV.a) } ${\mathcal L}_{1,3} \,=\,\langle e_{15},e_{14},
e_{13}\rangle$

The joint invariants of ${\mathcal L}_{1,3}$ are
\begin{equation}\nonumber
 t, \, {x}^{2}+{y}^{2}+{z}^{2},\, u \\
\end{equation}
so that the corresponding similarity transformations
\begin{equation}
 p\,=\,t, \, q\,=\, {x}^{2}+{y}^{2}+{z}^{2}, \, w(p,q)\,=\,u \\
\end{equation}
transform wave equation \eqref{eq-w} to
\begin{equation}
4\, \left( {\frac {\partial ^{2}}{\partial {q}^{2}}}w \left( p,q
 \right) \right) q+6\,{\frac {\partial }{\partial q}}w \left( p,q
 \right) -{\frac {\partial ^{2}}{\partial {p}^{2}}}w \left( p,q
 \right) =0
\end{equation}
which has the solution
\begin{equation}
w \left( p,q \right) ={\frac {{\it F_1} \left( \sqrt {q}+p \right)
+{ \it F_2} \left( -\sqrt {q}+p \right) }{\sqrt {q}}}.
\end{equation}
This leads to solution
\begin{equation}
u \left( {\it t},x,y,z \right) ={\frac {{\it F_1} \left( \sqrt {{x}^
{2}+{y}^{2}+{z}^{2}}+{\it t} \right) +{\it F_2} \left( -\sqrt {{x}^{
2}+{y}^{2}+{z}^{2}}+{\it t} \right) }{\sqrt
{{x}^{2}+{y}^{2}+{z}^{2}} }}
\end{equation}
of wave equation \eqref{eq-w}. \\

\noindent \textbf{(IV.b) } ${\mathcal L}_{2,3} \,=\,\langle e_7 -e_{16},e_{8},
e_{9}\rangle$

The joint invariants of ${\mathcal L}_{2,3}$ are
\begin{equation}\nonumber
 {\frac {y}{x}},\, {\frac {z}{x}}, \, u x \\
\end{equation}
so that the corresponding similarity transformations
\begin{equation}
 p\,=\,{\frac {y}{x}},\, q\,=\, {\frac {z}{x}},\, w(p,q)\,=\,u x\\
\end{equation}
transform wave equation \eqref{eq-w} to
\begin{equation}
\left( {\frac {\partial ^{2}}{\partial {p}^{2}}}w
 \right) {p}^{2}+2\,pq{\frac {\partial ^{2}}{\partial q\partial p}}w
 + \left( {\frac {\partial ^{2}}{\partial {q}^{2}}}
w \right) {q}^{2}+4\,p{\frac {\partial }{\partial p}}w +4\,
\left( {\frac {\partial }{\partial q}}w
 \right) q+{\frac {\partial ^{2}}{\partial {p}^{2}
}}w +2\,w +{\frac {\partial ^{2} }{\partial {q}^{2}}}w =0
\end{equation}
which has the solution
\begin{equation}
w \left( p,q \right) ={\frac {{\it C_1}\,p+{\it C_2}}{{p}^{2}+1}}+{
\frac {{\it C_3}\,q+{\it C_4}}{{q}^{2}+1}}.
\end{equation}
This leads to solution
\begin{equation}
u \left( {\it t},x,y,z \right) ={\frac {1}{x} \left( {\frac { \left(
{\it C_1}\,y+{\it C_2}\,x \right) x}{{x}^{2}+{y}^{2}}}+{\frac {
 \left( {\it C_3}\,z+{\it C_4}\,x \right) x}{{x}^{2}+{z}^{2}}}
 \right) }
\end{equation}
of wave equation \eqref{eq-w}. \\

\noindent \textbf{(IV.c) } ${\mathcal L}_{3,3}
\,=\,\langle e_1 + 2e_{10} + 2 e_{14}, -e_{3}-2 e_{11} + 2 e_{13}, -4 e_5 - 8 e_{12} - 8 e_{15}\rangle$ \\

The joint invariants of ${\mathcal L}_{3,3}$ are
\begin{equation}\nonumber
 {\frac {{x}^{2}+{y}^{2}+{z}^{2}-{{\it t}}^{2}+4}{{\it t}}} ,\, u t \\
\end{equation}
so that the corresponding similarity transformations
\begin{equation}
 p\,=\,{\frac {{x}^{2}+{y}^{2}+{z}^{2}-{{\it t}}^{2}+4}{{\it t}}}, \, w(p)\,=\, u t \\
\end{equation}
transform wave equation \eqref{eq-w} to
\begin{equation}
 \left( {\frac {{\rm d}^{2}}{{\rm d}{p}^{2}}}w \left( p \right)
 \right) {p}^{2}+4\, \left( {\frac {\rm d}{{\rm d}p}}w \left( p
 \right) \right) p+2\,w \left( p \right) +16\,{\frac {{\rm d}^{2}}{
{\rm d}{p}^{2}}}w \left( p \right) =0
\end{equation}
which has the solution
\begin{equation}
w \left( p \right) ={\frac {{\it C_1}\,p+{\it C_2}}{{p}^{2}+16}}
\end{equation}
This leads to solution
\begin{equation}
u \left( {\it t},x,y,z \right) = {\frac {-{\it C_1}\,{{\it t}}^{2}+{
\it C_2}\,{\it t}+{\it C_1}\, \left( {x}^{2}+{y}^{2}+{z}^{2}+4
 \right) }{{{\it t}}^{4}+ \left( -2\,{x}^{2}-2\,{y}^{2}-2\,{z}^{2}+8
 \right) {{\it t}}^{2}+ \left( {x}^{2}+{y}^{2}+{z}^{2}+4 \right) ^{2}
}}
\end{equation}
of wave equation \eqref{eq-w}. \\

\section*{Appendices}

\begin{center}
\textbf{Appendix 1} \\
\end{center}

\begin{table}[h]\label{T32}
 \caption{Commutator table for symmetry algebra of wave equation on Minkowski spacetime}

 \scalebox{0.7}{
\begin{tabular}{|c|c|c|c|c|c|c|c|c|c|c|c|c|c|c|c|c|}
 \hline
 \empty & $X_{1}$ & $X_{2}$ & $X_{3}$ & $X_{4}$ & $X_{5}$ & $X_{6}$ & $X_{7}$& $X_{8}$ & $X_{9}$ & $X_{10}$ & $X_{11}$ & $X_{12}$ & $X_{13}$ & $X_{14}$ & $X_{15}$ & $X_{16}$ \\
 \hline
 $X_{1}$ & 0 & $\frac{-1}{2}X_{9}$ & 0 & 0 & 0 & 0& $-X_{1}$ & $-X_{2}$ & 0 & $X_{16}-X_{7}$ & $X_{15}$ & $X_{13}$ & $-X_{5}$ & 0 & $-X_{3}$ & 0 \\
 \hline
 $X_{2}$ & \empty & 0 & 0 & $-X_{15}$ & 0 & $-X_{13}$ & 0 & $-X_{10}$ & $2X_{1}$ & $-X_{8}$ & 0 & 0 & $-X_{6}$ & 0 & $-X_{4}$ & 0 \\
 \hline
 $X_{3}$ & \empty & \empty & 0 & $\frac{-1}{2}X_{9}$ & 0 & 0 & $-X_{3}$ & $-X_{4}$ & 0 & $-X_{15}$ & $X_{16}-X_{7}$ & $X_{14}$ & 0 & $-X_{5}$ & $X_{1}$ & 0 \\
 \hline
 $X_{4}$ & \empty& \empty & \empty & 0 & 0 & $-X_{14}$ & 0 & $-X_{11}$ & $2X_{3}$ & 0 & $-X_{8}$ & 0 & 0 & $-X_{6}$ & $X_{2}$ & 0 \\
 \hline
 $X_{5}$ & \empty & \empty & \empty & \empty & 0 & $\frac{-1}{2}X_{9}$ & $-X_{5}$ & $-X_{6}$ & 0 & $-X_{13}$ & $-X_{14}$ & $X_{16}-X_{7}$ & $X_{1}$ & $X_{3}$ & 0 & 0 \\
 \hline
 $X_{6}$ & \empty & \empty & \empty & \empty & \empty & 0 & 0 & $-X_{12}$ & $2X_{5}$ & 0 & 0 & $-X_{8}$ & $X_{2}$ & $X_{4}$ & 0 & 0 \\
 \hline
$X_{7}$ & \empty & \empty & \empty & \empty & \empty & \empty & 0 & $-X_{8}$ & $X_{9}$ & $-X_{10}$ & $-X_{11}$ & $-X_{12}$ & 0 & 0 & 0 & 0 \\
 \hline
 $X_{8}$ & \empty & \empty & \empty & \empty & \empty & \empty & \empty & 0 & $2X_{7}-2X_{16}$ & 0 & 0 & 0 & 0 & 0 & 0 & 0 \\
 \hline
 $X_{9}$ & \empty & \empty & \empty & \empty & \empty & \empty & \empty & \empty &0 & $-2X_{2}$ & $-2X_{4}$ & $-2X_{6}$& 0 & 0 & 0 & 0 \\
 \hline
 $X_{10}$ & \empty & \empty & \empty & \empty & \empty & \empty & \empty & \empty & \empty & 0 & 0 & 0 & $-X_{12}$ & 0 & $-X_{11}$ & 0 \\
 \hline
 $X_{11}$ & \empty & \empty & \empty & \empty & \empty & \empty & \empty & \empty & \empty & \empty & 0 & 0 & 0 & $-X_{12}$ & $X_{10}$ & 0 \\
 \hline
 $X_{12}$ & \empty & \empty & \empty & \empty & \empty & \empty & \empty & \empty & \empty & \empty & \empty & 0 & $X_{10}$ & $X_{11}$ & 0 & 0 \\
 \hline
 $X_{13}$ & \empty & \empty & \empty & \empty & \empty & \empty & \empty & \empty & \empty & \empty & \empty & \empty & 0 & $X_{15}$ & $-X_{14}$ & 0 \\
 \hline
 $X_{14}$ & \empty & \empty & \empty & \empty & \empty & \empty & \empty & \empty & \empty & \empty & \empty & \empty & \empty & 0 & $X_{13}$ & 0 \\
 \hline
 $X_{15}$ & \empty & \empty & \empty & \empty & \empty & \empty & \empty & \empty & \empty & \empty & \empty & \empty & \empty & \empty & 0 & 0 \\
 \hline
$X_{16}$ & \empty & \empty & \empty & \empty & \empty & \empty & \empty & \empty & \empty & \empty & \empty & \empty & \empty & \empty & \empty & 0 \\
 \hline
\end{tabular}}
where $[X_{i},X_{j}]\,=\,-[X_{j},X_{i}]$.
\end{table}

\begin{center}
\textbf{Appendix 2} \\
\end{center}

\begin{table}[h]\label{T32b}
 \caption{Commutator table for $G_2$}

 \scalebox{0.7}{
\begin{tabular}{|c|c|c|c|c|c|c|c|c|c|c|c|c|c|c|}
 \hline
 \empty & $X_{1}$ & $X_{2}$ & $X_{3}$ & $X_{4}$ & $X_{5}$ & $X_{6}$ & $X_{7}$& $X_{8}$ & $X_{9}$ & $X_{10}$ & $X_{11}$ & $X_{12}$ & $X_{13}$ & $X_{14}$ \\
 \hline
 $X_{1}$ & 0 & 0 & 0 & 0 & 0 & $-X_1$& 0 & 0 & 0 & $-\frac{1}{2}X_{2}$ & $-\frac{1}{2}X_{3}$ & $-\frac{1}{2}X_{4}$ & $-\frac{1}{2}X_{5}$ & $-X_6$ \\
 \hline
 $X_{2}$ & \empty & 0 & 0 & 0 & $4X_{1}$ & $-\frac{1}{2}X_{2}$ & 0 & $\frac{3}{2}X_{2}$ & $-X_{3}$ & 0 & 0 & $-\frac{4}{3}X_7 $ & $-2X_8+2X_6$ & $-X_{10}$ \\
 \hline
 $X_{3}$ & \empty & \empty & 0 & $-4X_{1}$ & 0 & $-\frac{1}{2}X_{3}$ & $-\frac{3}{2}X_2$ & $\frac{1}{2}X_{3}$ & $-X_4$ & 0 & $\frac{4}{3}X_{7}$ & $\frac{2}{3}X_8-2X_6$ & $2X_{9}$ & $-X_{11}$ \\
 \hline
 $X_{4}$ & \empty & \empty & \empty & 0 & 0 & $-\frac{1}{2}X_{4}$ & $-2X_3$ & $-\frac{1}{2}X_{4}$ & $-X_{5}$ & $-\frac{4}{3}X_7$ & $\frac{2}{3}X_8+2X_6$ & $-\frac{8}{3}X_9$ & $-X_{12}$ & 0 \\
 \hline
 $X_{5}$ & \empty & \empty & \empty & \empty & 0 & $-\frac{1}{2}X_{5}$ & $-\frac{3}{2}X_{4}$ & $-\frac{3}{2}X_{5}$ & 0 & $-2X_{8}-2X_6$ & $2X_{9}$ & 0 & 0 & $-X_{13}$\\
 \hline
 $X_{6}$ & \empty & \empty & \empty & \empty & \empty & 0 & 0 & 0 & 0 & $-\frac{1}{2}X_{10}$ & $-\frac{1}{2}X_{11}$ & $-\frac{1}{2}X_{12}$ & $-\frac{1}{2}X_{13}$ & $-X_{14}$ \\
 \hline
$X_{7}$ & \empty & \empty & \empty & \empty & \empty & \empty & 0 & $X_{7}$ & $-X_{8}$ & 0 & $\frac{3}{2}X_{10}$ & $2X_{11}$ & $\frac{3}{2}X_{12}$ & 0 \\
 \hline
 $X_{8}$ & \empty & \empty & \empty & \empty & \empty & \empty & \empty & 0 & $X_{9}$ & $-\frac{3}{2}X_{10}$ & $-\frac{1}{2}X_{11}$ & $\frac{1}{2}X_{12}$ & $\frac{3}{2}X_{13}$ & 0 \\
 \hline
 $X_{9}$ & \empty & \empty & \empty & \empty & \empty & \empty & \empty & \empty & 0 & $X_{11}$ & $X_{12}$ & $X_{13}$& 0 & 0 \\
 \hline
 $X_{10}$ & \empty & \empty & \empty & \empty & \empty & \empty & \empty & \empty & \empty & 0 & 0 & 0 & $2X_{14}$ & 0 \\
 \hline
 $X_{11}$ & \empty & \empty & \empty & \empty & \empty & \empty & \empty & \empty & \empty & \empty & 0 & $-2X_{14}$ & 0 & 0 \\
 \hline
 $X_{12}$ & \empty & \empty & \empty & \empty & \empty & \empty & \empty & \empty & \empty & \empty & \empty & 0 & 0 & 0 \\
 \hline
 $X_{13}$ & \empty & \empty & \empty & \empty & \empty & \empty & \empty & \empty & \empty & \empty & \empty & \empty & 0 & 0 \\
 \hline
 $X_{14}$ & \empty & \empty & \empty & \empty & \empty & \empty & \empty & \empty & \empty & \empty & \empty & \empty & \empty & 0 \\
 \hline

 \hline
\end{tabular}}

where $[X_{i},X_{j}]\,=\,-[X_{j},X_{i}]$.
\end{table}

\section*{Acknowledgements}

We thank the two referees for helpful comments. One of us (HA) thanks Karl-Hermann Neeb 
for a very helpful correspondence. The third-author is supported by a J. C. Bose 
Fellowship.

%%%%%%%%%%%%%%%%%%%%%%%%%%%%%%%%%%%%%%%%%%%%%%%%%%%%%%%%%%%%%%%

\end{document}